\documentclass[oneside]{amsart}
\usepackage{amsmath, amsfonts, amsthm, graphicx,amssymb, bold-extra,enumerate,cite,color,multicol,longtable,tabu,helvet,setspace}
\usepackage{hyperref}

\usepackage[square,sort,comma,numbers]{natbib}
\newcommand{\Z}{\mathbb{Z}}

\newcommand{\N}{\mathbb{N}}
\newcommand{\Q}{\mathbb{Q}}

\newtheorem{thm}{Theorem}

\newtheorem{lma}{Lemma}

\theoremstyle{definition}

\newtheorem{dfn}{Definition}
\newtheorem{rmk}{\sc Remark}

\newcommand\T{\rule{0pt}{2.6ex}}

\newcommand\TT{\rule{0pt}{3.7ex}}
\newcommand\B{\rule[-1.2ex]{0pt}{0pt}}

\title[Euclidean Ideals]{Two Classes of Number Fields with a Non-Principal Euclidean Ideal}
\author{Catherine Hsu}
\thanks{The author's research is partially supported by a GAANN Fellowship}

\begin{document}
\maketitle

\begin{abstract} This paper introduces two classes of totally real quartic number fields, one of biquadratic extensions and one of cyclic extensions, each of which has a non-principal Euclidean ideal. It generalizes  techniques of Graves used to prove that the number field $\Q(\sqrt{2},\sqrt{35})$ has a non-principal Euclidean ideal. 
\end{abstract}
\section{Introduction}
\subsection{Notation} Given a number field $K,$ we will denote its ring of integers $\mathcal{O}_K,$ its class group $\mathrm{Cl}_K,$ its class number $h_K,$ and its conductor $f(K)$. When $K/\Q$ is abelian, we will denote its Hilbert class field over $\Q$ by $H/\Q$. In this situation, for a rational prime $p\in\Z$, we will use $\mathfrak{p}$ to denote a prime in $\mathcal{O}_K$ lying over $(p)$  and $\mathfrak{P}$ to denote a prime in $\mathcal{O}_H$ lying over $\mathfrak{p}$.
\subsection{Background and Main Results} In 1979, Lenstra \cite{lenstra1979euclidean} defined the Euclidean ideal, a generalization of the Euclidean algorithm:
\begin{dfn} Suppose $R$ is a Dedekind domain and that $E$ is the set of fractional ideals that contain $R$. If $C$ is an ideal of $R,$ it is \textit{Euclidean} if there exists a function $\psi:E\to\,W,$ $W$ a well-ordered set, such that for all $I\in E$ and all $x\in IC\,\backslash C,$ there exists some $y\in C$ such that \[\psi((x+y)^{-1}IC)<\psi(I).\]
We say $\psi$ is a Euclidean algorithm for $C$ and $C$ is a Euclidean ideal.
\end{dfn}
While the existence of a Euclidean ideal $C$ in a ring of integers $\mathcal{O}_K$ does not give a method for computing greatest common divisors as a Euclidean algorithm would, it does guarantee that the class group has a certain structure, namely $\mathrm{Cl}_K = \langle [C]\rangle$.

One method that can be used in certain situations to produce a Euclidean algorithm for an ideal $C$ is a Motzkin-type construction for ideals \cite{graves2013growth}. As its name suggests, this method is a generalization of Motzkin's construction \cite{motzkin1949euclidean}, which can be used to find a Euclidean algorithm on certain integral domains. Before we can give an explicit example of such a Euclidean algorithm for $C,$ we need to recall the definition from \cite{graves2013growth} of a Motzkin-type construction for ideals:

\begin{dfn} Given a non-zero ideal $C$ in $R,$ we define
\[\begin{split}
&A_{0,C}:=\{R\}
\\&A_{i,C}:=A_{i-1,C}\cup\left\{I\in E\;\vline\; \begin{split}\forall x\in IC\,\backslash C&,\, \exists y\in C \text{ such }\\\text{ that } (x+y)^{-1}&IC\in A_{i-1,C} \end{split}\right\} \text{ for } i>0,
\\&\text{and } A_C:=\bigcup_{i = 0}^\infty A_{i,C}.
\end{split}\]
Note that the $A_{i,C}$'s are nested sets.
\end{dfn}
Given this definition, Graves [2] proves that if $A_C=E,$ where $E$ is the set fractional ideals containing $R,$ we can construct a Euclidean algorithm for $C$ via the function $\psi_C:E\to\,\N$ defined by \[\psi_C(I) = i, \;\;\;\;\text{ if } I\in A_{i,C}\,\backslash A_{i-1,C}.\]
Indeed, if $I$ is an ideal in $E$ and $x\in IC\,\backslash\, C,$ then since $I\in A_C,$ there exists some $y\in C$ such that $(x+y)^{-1}IC$ is an element of $A_{\psi_C(I)-1,C},$ and hence, for this $y\in C,$  \[\psi_C((x+y)^{-1}IC)\leq \psi_C(I)-1<\psi_C(I).\]

Now, just as the existence of a Euclidean algorithm for the ring of integers $\mathcal{O}_K$ in a number field $K$ implies a trivial class group, the existence of a Euclidean ideal $C$ in $\mathcal{O}_K$ implies a cyclic class group with generator $[C]$. In fact, Lenstra proves a much stronger result:

\begin{thm}\label{lenstra} \cite{lenstra1979euclidean} (Lenstra, 1979) Suppose $K$ is a number field, $|\mathcal{O}_K^\times|=\infty,$ and $C$ is an ideal of $\mathcal{O}_K$. If one assumes the Generalized Riemann Hypothesis, then $C$ is a Euclidean ideal if and only if $\mathrm{Cl}_K =\langle[C]\rangle$.
\end{thm}

By generalizing the work of Harper and Murty \cite{mharper1,harper2004euclidean} on Euclidean rings to the Euclidean ideal case, Graves \cite{graves2013growth} proves that under certain conditions,  Theorem \ref{lenstra} holds without assuming the Generalized Riemann Hypothesis:

\begin{thm}\label{growthresult}\cite{graves2013growth} (Graves, 2013) If $K$ is a number field such that $|\mathcal{O}_K^\times|=\infty,$ if $[C]$ generates $\mathrm{Cl}_K,$ and if  \[\left|\left\{\begin{split}\text{prime }&\text{ideals }\\\mathfrak{p}\subset&\,\mathcal{O}_K\end{split}\,\vline\,\,\mathrm{Nm}(\mathfrak{p})\leq x,\,[\mathfrak{p}]=[C],\,\mathcal{O}_K^\times\twoheadrightarrow(\mathcal{O}_K/\mathfrak{p})^\times\right\}\right|\gg\frac{x}{\log^2x},\] then $C$ is a Euclidean ideal.
\end{thm}
In order to use the growth result given in  Theorem \ref{growthresult} to find explicit examples of number fields with a Euclidean ideal, we require  the following theorem which Graves stated in \cite{graves2011has} by utilizing results from \cite{heathbrown,narkiewicz1988units}:

\begin{thm}\label{growthresult2} If $K$ is a totally real number field with conductor $f(K),$ if $\{e_1,e_2,e_3\}$ is a multiplicatively independent set contained in $\mathcal{O}_K^\times,$ if $\ell =\mathrm{lcm}(16,f(K)),$ and if $(u,\ell) = \left(\frac{u-1}{2},\ell\right)=1,$ then \[\left|\left\{\begin{split}\text{prime }&\text{ideals }\mathfrak{p}\\\text{of first}&\text{ degree}\end{split}\,\vline\,\begin{split}\mathrm{Nm}(\mathfrak{p})&\equiv u\pmod \ell,\\\mathrm{Nm}(\mathfrak{p})\leq x, &\langle-1,e_i\rangle\twoheadrightarrow(\mathcal{O}/\mathfrak{p})^\times\end{split}\right\}\right|\gg\frac{x}{\log^2x},\] for at least one $i$.
\end{thm}

In \cite{graves2011has}, Graves uses Theorems \ref{growthresult} and \ref{growthresult2} to prove that $\Q(\sqrt{2},\sqrt{35})$ has a non-principal Euclidean ideal.  In this paper, we generalize Graves' work to prove our first main result that a certain class of biquadratic number fields have a non-principal Euclidean ideal. The first main result is:
\begin{thm}\label{newclass}If a number field $K$ is of the form $\mathbb{Q}(\sqrt{q},\sqrt{k\cdot r}),$ where  $q,\,k,\, r\geq 29$ are distinct rational primes satisfying $q,\,k,\,r\equiv 1\pmod 4,$ and if $h_K=2,$ then $K$ has a non-principal Euclidean ideal.
\end{thm}

Then by slightly modifying the techniques used to prove Theorem \ref{newclass}, we obtain our second main result that a certain class of cyclic number fields also have a non-principal Euclidean ideal. The second main result is:

\begin{thm}\label{newclass2}If a number field $K$ is of the form \[\mathbb{Q}\left(\sqrt{q(k+b\sqrt{k})}\right),\] where  $q,\,k\geq 17$ are distinct rational primes satisfying $q,\,k\equiv 1\pmod 4$ and $b>0$ is an integer satisfying $b\equiv 0\pmod 4,$ if $k-b^2>0$ is a perfect square, and if $h_K=2,$ then $K$ has a non-principal Euclidean ideal.
\end{thm}

\begin{rmk} The significance of these results is that while Graves introduced one number field with a non-principal Euclidean ideal, we give conditions that provide two new classes of examples of number fields, each of which contains a non-principal Euclidean ideal. In particular, Murty and Graves \cite{graves2013family} provide the only other examples of number fields with a non-principal Euclidean ideal found without assuming the Generalized Riemann Hypothesis. Their results require unit rank at least 4; our results do not require this condition.
\end{rmk}
\begin{rmk}
Using PARI \cite{PARI2}, the author has found hundreds of number fields satisfying the conditions given in Theorems \ref{newclass} and \ref{newclass2}, some of which are given in Tables 1 and 2. The author conjectures that both of these classes of number fields with a non-principal Euclidean ideal are in fact infinite. 
\end{rmk}
\subsection{Acknowledgments} I would like to thank my advisor Ellen Eischen for suggesting this project as well as for providing useful feedback. I would also like to thank Dylan Muckerman for helping me write the Python code used to find explicit examples of number fields with a non-principal Euclidean ideal.

\section{Proof of Main Results}

In order to prove our main results, we would like to apply Theorem \ref{growthresult2} to the number field $K$ in order to obtain a growth result which satisfies the hypotheses of Theorem \ref{growthresult}. We will first prove our result in the biquadratic case and then modify this proof to obtain our second result in the cyclic case.
\subsection{Main Result in the Biquadratic Case} Throughout this section, we will assume that $K$ is a number field satisfying the conditions given in  Theorem \ref{newclass}, i.e., $h_K=2$ and \[K = \Q(\sqrt{q},\sqrt{k\cdot r}),\] where $q,\,k,\,r\geq 29$ are distinct rational primes satisfying $q,\,k,\,r\equiv 1\pmod 4$. 

Before we can prove  Theorem \ref{newclass}, we need the following two lemmas.
\begin{lma}\label{conductor} The conductor $f(K)$ of  $K$ is $qkr$.
\end{lma}

\begin{proof} Since both $q\equiv 1\pmod 4$ and $k\cdot r\equiv 1 \pmod 4,$ the conductors of $\Q(\sqrt{q}),$ $\Q(\sqrt{k\cdot r})$ are $q$ and $kr,$ respectively. So, since $\Q(\zeta_{f(K)})$ must contain both of these quadratic fields,  the minimality of a conductor implies that \[\Q(\zeta_q),\Q(\zeta_{k r})\subseteq\Q(\zeta_{f(K)}).\] By the theory of cyclotomic fields, both $q$ and $kr$ must divide $f(K),$ and hence, since \[K\subseteq \Q(\sqrt{q})\Q(\sqrt{k\cdot r})\subseteq\Q(\zeta_q)\Q(\zeta_{kr})\subseteq\Q(\zeta_{qk r}),\] we conclude that $f(K)=qk r$.
\end{proof}

\begin{lma}\label{hilbert}
The Hilbert class field of $K$ over $\Q$ is $H = \Q(\sqrt{q},\sqrt{k},\sqrt{ r})$. 
\end{lma}

\begin{proof} Since $[H:K] = 2$ and $K$ has class number 2 by hypothesis, it is sufficient to show that $H/K$ is an unramified extension. Indeed, since $q,\,k,\, r\equiv 1\pmod 4,$ the conductors of $\Q(\sqrt{q}),$ $\Q(\sqrt{k}),\,\Q(\sqrt{ r})$ are $q,\,k,$ and $r,$ respectively. Then, as in Lemma 1, \[\Q(\zeta_q),\Q(\zeta_{k }),\Q(\zeta_r)\subseteq\Q(\zeta_{f(H)})\] so that $q, k,$ and $r$ each divides $f(H)$.  Hence, since \[H\subseteq\Q(\sqrt{q})\Q(\sqrt{k})\Q(\sqrt{r})\subseteq\Q(\zeta_q)\Q(\zeta_k)\Q(\zeta_r)\subseteq \Q(\zeta_{qkr}),\] we again conclude that $f(H) = qk r$. Therefore, only prime ideals lying over $(q),\,(k),$ or $(r)$  can ramify in $H/K,$ and we now show that the ramification index of each of these prime ideals is 2 in both $H/\Q$ and $K/\Q$. 

Let $F =\Q(\sqrt{k\cdot r}),$ $L = \Q(\sqrt{k},\sqrt{ r}),$ and suppose that $\mathfrak{q},\,\mathfrak{Q},\,\mathfrak{b}$ and $\mathfrak{B}$ lie over $(q)$ in $K/\Q,\,H/\Q,\,F/\Q,$ and $L/\Q,$ respectively. Since $q$ does not divide  $k$ or $r,$ the prime ideal $(q)$ does not ramify in $\Q(\sqrt{k})$ or $\Q(\sqrt{ r}),$ and hence, since $L =\Q(\sqrt{k})\Q(\sqrt{ r}),$ we have $e(\mathfrak{B}/q)=1$. But, since $q$ divides $f(H),$ $(q)$ must ramify in $H/\Q$ so that \[1<e(\mathfrak{Q}/q)=e(\mathfrak{Q}/\mathfrak{B})\cdot e(\mathfrak{B}/q)\leq 2\cdot 1\Rightarrow e(\mathfrak{Q}/q)=2.\] Now, as $F=\Q(\sqrt{k\cdot r})\subset L,$  the prime ideal $(q)$ also does not ramify in $\Q(\sqrt{k\cdot r})$. But again, since $q$ divides $f(K),$ $(q)$ ramifies in $K/\Q$ so that \[1<e(\mathfrak{q}/q)=e(\mathfrak{q}/\mathfrak{b})\cdot e(\mathfrak{b}/q)\leq 2\cdot 1\Rightarrow e(\mathfrak{q}/q)=2.\]  Thus, since \[2=e(\mathfrak{Q}/q) = e(\mathfrak{Q}/\mathfrak{q})\cdot e(\mathfrak{q}/q)=e(\mathfrak{Q}/\mathfrak{q})\cdot 2,\] we conclude that $e(\mathfrak{Q}/\mathfrak{q})=1$ so that $(q)$ is unramified in $H/K$. 

To see that the prime ideal $(k)$ is unramified in $H/K,$ we apply the same argument with $F=\Q(\sqrt{q})$ and $L=\Q(\sqrt{q},\sqrt{r}),$ and a similar modification gives that the prime ideal $(r)$ is also unramified in $H/K$. Hence, since $H/K$ is an unramified extension of degree 2, $H$ is the Hilbert class field of $K$ over $\Q$.
\end{proof}

We are now ready to prove the first main result of this paper:

\begin{proof}[Proof of  Theorem \ref{newclass}] Our goal is to apply  Theorem \ref{growthresult2} to $K=\Q(\sqrt{q},\sqrt{k\cdot r})$ in order to obtain the growth result necessary to apply Theorem \ref{growthresult}. Since $K$ is a totally real number field of degree 4, Dirichlet's unit theorem gives that we can find a set $\{e_1,e_2,e_3\}$ of multiplicatively independent elements in $\mathcal{O}_K^\times$. So, the hypotheses of  Theorem \ref{growthresult2} will be met if, for $\ell = \mathrm{lcm}(16,f(K))=16\cdot qkr,$ we can find some $u\in\Z$ satisfying the following conditions:
\begin{enumerate}[(1)]
\item $(u,\ell) =1;$
\item $\left(\frac{u-1}{2},\ell\right) = 1$.
\end{enumerate}

 To translate the growth result given by  Theorem \ref{growthresult2} into a growth result satisfying the hypotheses of Theorem \ref{growthresult}, we need $u\in\Z$ to satisfy one additional condition:
\begin{enumerate}[(3)]
\item For primes $\mathfrak{P},\mathfrak{p}$ lying over any $(p)$ such that $p\equiv u\pmod{qk r},$ we have that $\mathfrak{P}/p$ has residue degree 2 and $\mathfrak{p}/p$ has residue degree 1.
\end{enumerate}
Note that the last condition ensures that $\mathfrak{P}/\mathfrak{p}$ has residue degree 2 so that by Artin reciprocity, any prime ideal $\mathfrak{p}$ satisfying Condition (3) is non-principal.

In order to find such an element $u\in\Z,$ we first determine a set of conditions equivalent to Condition (2). Indeed, we have \[\begin{split}
u \text{ satisfies Condition } (2)\Leftrightarrow\, &u\not\equiv 1\pmod q,\\& u\not\equiv 1\pmod k,\\& u\not\equiv 1\pmod  r,\\&u\not\equiv 1\pmod 4.
\end{split}\]
So, consider the sets \[\begin{split} A_q &= \{\text{primes }p\in\Z: p=1+nq,\;\;n\in\Z\},\\A_k&=\{\text{primes }p\in\Z: p=1+mk,\;\;m\in\Z\},\\A_r&= \{\text{primes }p\in\Z: p=1+tr,\;\;t\in\Z\}.\end{split}\] By the prime number theorem for arithmetic progressions, the primes in the set $A = A_q\cup A_k\cup A_r$ have density bounded by \[\frac{1}{q-1}+\frac{1}{k-1}+\frac{1}{ r-1},\] and this is strictly less than $\frac{1}{8}$ since $q,k, r\geq 29$ by hypothesis.

Next, for Condition (3), we see that \[\text{primes } \mathfrak{P},\mathfrak{p}\text{ lying over } (p) \text{ satisfy Condition (3)}\Leftrightarrow \left(\frac{p}{K/\Q}\right)= 1,\, \left(\frac{p}{H/\Q}\right)\neq 1.\] By the Chebotarev density theorem, the density of the set \[T_K=\left\{p\in\Z:\left(\frac{p}{K/\Q}\right)= 1\right\}\] is $\frac{1}{4}$ while the density of the set
\[T_H=\left\{p\in\Z:\left(\frac{p}{H/\Q}\right)= 1\right\}\] is $\frac{1}{8}$. So, since any prime in $T_K\backslash T_H$ satisfies Condition (3), the set of primes satisfying Condition (3) is at least $\frac{1}{8}$. Hence, we may choose a prime $s\in T_K\backslash T_H,$ distinct from $2,q,k,$ and $ r,$ that is not contained in $A$. 

It remains to address the condition $u\not\equiv 1 \pmod 4$ in Condition (2): 
\begin{itemize}
\item If $s\not\equiv 1\pmod 4,$ then choose $u = s$. 
\item If $s\equiv 1\pmod 4,$ then choose $u = s+2qkr$. 
\end{itemize}

We now verify our choice of $u$ satisfies the required Conditions (1)-(3):
\begin{enumerate}[(1)]
\item $(u,\ell) = 1$: Clearly, $2,q,k,$ and $r$ are the only primes dividing $\ell=16\cdot qkr$. If $u = s,$ then since $s$ is a prime distinct from $2,q,k$ and $r,$ this condition is satisfied. If $u = s+2qkr,$ then this condition is also satisfied -- if this were not the case, then either $2,q,k$ or $r$ would divide $s,$ a contradiction. 
\item $\left(\frac{u-1}{2},\ell\right) = 1$: By choice, $u\not\equiv 1\pmod 4$ so that $2$ does not divide $\frac{u-1}{2}$. Since $u\equiv s\pmod{qkr},$ where $s\not\equiv 1\pmod q,$ $s\not\equiv 1\pmod k,$ and $s\not\equiv 1\pmod r$ by choice, neither $q,k,$ nor $r$ divides $u-1$ so that this condition is satisfied.
\item As the Artin map factors through $f(H)=f(K) = qkr,$ the choice of $s$ above guarantees that primes $\mathfrak{P},\mathfrak{p}$ lying over any prime ideal $(p)$ such that \[p\equiv u\equiv s\pmod{qkr}\] will satisfy this condition. Note that this step requires that $H/\Q$ is abelian.
\end{enumerate}

Thus, we may apply Theorem \ref{growthresult2} with our choice of $u$ to obtain

\[\left|\left\{\begin{split}\text{prime }&\text{ideals }\mathfrak{p}\\\text{of first}&\text{ degree}\end{split}\,\vline\,\begin{split}\mathrm{Nm}(\mathfrak{p})&\equiv u\pmod {\ell},\\\mathrm{Nm}(\mathfrak{p})\leq x, &\langle-1,e_i\rangle\twoheadrightarrow(\mathcal{O}/\mathfrak{p})^\times\end{split}\right\}\right|\gg\frac{x}{\log^2x}\]
so that since $\mathcal{O}_K^\times\twoheadrightarrow \langle-1,e_i\rangle$ for any $e_i\in\mathcal{O}_K^\times,$ this implies

\[\left|\left\{\begin{split}\text{prime }&\text{ideals }\mathfrak{p}\\\text{of first}&\text{ degree}\end{split}\,\vline\,\begin{split}\mathrm{Nm}(\mathfrak{p})&\equiv u\pmod {\ell},\\\mathrm{Nm}(\mathfrak{p})\leq x, &\mathcal{O}_K^\times\twoheadrightarrow(\mathcal{O}/\mathfrak{p})^\times\end{split}\right\}\right|\gg\frac{x}{\log^2x}.\] By Condition (3), each of these primes is non-principal, and hence, if $\mathrm{Cl}_K = \langle[C]\rangle,$  \[\left|\left\{\begin{split}\text{prime }&\text{ideals }\\\mathfrak{p}\subset&\,\mathcal{O}_K\end{split}\,\vline\,\,\mathrm{Nm}(\mathfrak{p})\leq x,\,[\mathfrak{p}]=[C],\,\mathcal{O}_K^\times\twoheadrightarrow(\mathcal{O}_K/\mathfrak{p})^\times\right\}\right|\gg\frac{x}{\log^2x}.\] Hence, by  Theorem \ref{growthresult}, $C$ is a Euclidean ideal.
\end{proof}

\subsection{Main Result in the Cyclic Case} We will now assume that $K$ is a number field satisfying the conditions given in  Theorem \ref{newclass2}, i.e., $h_K=2$ and \[K=\mathbb{Q}\left(\sqrt{q(k+b\sqrt{k})}\right),\] where  $q,\,k\geq 17$ are distinct rational primes satisfying $q,\,k\equiv 1\pmod 4,$ $b>0$ is an integer satisfying $b\equiv 0\pmod 4,$ and $k-b^2>0$ is a perfect square.

As before, in order to prove  Theorem \ref{newclass2}, we first need two lemmas which are analogous to Lemmas \ref{conductor} and \ref{hilbert} in the biquadratic case.

\begin{lma}\label{conductor2} The conductor $f(K)$ of  $K$ is $qk$.
\end{lma}

\begin{proof} Since $q,k\geq 17$ are distinct rational primes satisfying $q,\,k\equiv 1\pmod 4,$ $b>0$ is an integer satisfying $b\equiv 0\pmod 4,$ and $k-b^2>0$ is a perfect square by hypothesis, we have that $f(K)=qk$ by the main result of \cite{spearman1997conductor}.
\end{proof}

\begin{lma}\label{hilbert2}
The Hilbert class field of $K$ over $\Q$ is \[H = \Q\left(\sqrt{q},\sqrt{k+b\sqrt{k}}\right).\] 
\end{lma}

\begin{proof} Since $[H:K] = 2$ and $K$ has class number 2 by hypothesis, it is sufficient to show that $H/K$ is an unramified extension. Indeed, since $q\equiv 1\pmod 4,$ the conductor of $\Q(\sqrt{q})$ is $q,$ and since $k\equiv 1\pmod 4$ and $b\equiv 0\pmod 4,$ the main result of \cite{spearman1997conductor} gives that the conductor of $\Q(\sqrt{k+b\sqrt{k}})$ is $k$. Then, as in Lemma \ref{hilbert}, \[\Q(\zeta_q),\Q(\zeta_{k})\subseteq\Q(\zeta_{f(H)})\] so that $q$ and $k$ each divides $f(H)$.  Hence, since \[H\subseteq\Q(\sqrt{q})\Q\left(\sqrt{k+b\sqrt{k}}\right)\subseteq\Q(\zeta_q)\Q(\zeta_k)\subseteq \Q(\zeta_{qk}),\] we conclude that $f(H) = qk$. Therefore, only prime ideals lying over $(q)$ or $(k)$ can ramify in $H/K,$ and we now show that the ramification indices of the prime ideals $(q)$ and $(k)$ are 2 and 4, respectively, in both $H/\Q$ and $K/\Q$. 

We first show that the prime ideal $(q)$ has ramification index 2 in both $H/\Q$ and $K/\Q$. Let $F = \Q(\sqrt{k})$ and $L =\Q(\sqrt{k+b\sqrt{k}}),$ and suppose that $\mathfrak{q},\,\mathfrak{Q},\,\mathfrak{b}$ and $\mathfrak{B}$ lie over $(q)$ in $K/\Q,\,H/\Q,\,F/\Q,$ and $L/\Q,$ respectively. Since $q$ does not divide  $f(L)=k,$ the prime ideal $(q)$ does not ramify in $L$ so that $e(\mathfrak{B}/q)=1$. But, since $q$ divides $f(H),$ $(q)$ must ramify in $H/\Q$ so that \[1<e(\mathfrak{Q}/q)=e(\mathfrak{Q}/\mathfrak{B})\cdot e(\mathfrak{B}/q)\leq 2\cdot 1\Rightarrow e(\mathfrak{Q}/q)=2.\] Now, as $F=\Q(\sqrt{k})\subset K$ also has conductor $k,$  the prime ideal $(q)$ also does not ramify in $\Q(\sqrt{k})$. But again, since $q$ divides $f(K),$ $(q)$ ramifies in $K/\Q$ so that \[1<e(\mathfrak{q}/q)=e(\mathfrak{q}/\mathfrak{b})\cdot e(\mathfrak{b}/q)\leq 2\cdot 1\Rightarrow e(\mathfrak{q}/q)=2.\]  Thus, since \[2=e(\mathfrak{Q}/q) = e(\mathfrak{Q}/\mathfrak{q})\cdot e(\mathfrak{q}/q)=e(\mathfrak{Q}/\mathfrak{q})\cdot 2,\] we conclude that $e(\mathfrak{Q}/\mathfrak{q})=1$ so that $(q)$ is unramified in $H/K$. 

It remains to show that the prime ideal $(k)$ has ramification index 4 in both $H/\Q$ and $K/\Q$. In addition to the fields defined above, let $F' = \Q(\sqrt{q}),$ and suppose that $\mathfrak{k},\,\mathfrak{K},\,\mathfrak{c}$ and $\mathfrak{c'}$ lie over $(k)$ in $K/\Q,\,H/\Q,\,F/\Q$ and $F'/\Q,$ respectively. Since $k$ divides $f(F)=k,$ we see that $e(\mathfrak{c}/k)=2$. Moreover, we have $K=\Q(\alpha),$  where \[\alpha = \sqrt{q(k+b\sqrt{k})}\] has minimal polynomial $f(x)=x^2-q(k+b\sqrt{k})$ over $F$. Since $f(x)\in\mathcal{O}_F[x]$ is an Eisenstein polynomial at the prime ideal $(\sqrt{k}),$ we see that $e(\mathfrak{k}/\mathfrak{c})=2,$ and hence,  \[e(\mathfrak{k}/k) = e(\mathfrak{k}/\mathfrak{c})\cdot e(\mathfrak{c}/k)=2\cdot 2=4.\]
Lastly, since $k$ does not divide  $f(F')=q,$ the prime ideal $(k)$ does not ramify in $F'$ so that $e(\mathfrak{c'}/k)=1$. But, since $K\subset H,$ we see that $e(\mathfrak{k}/k) =4$ implies \[4\leq e(\mathfrak{K}/k)=e(\mathfrak{K}/\mathfrak{c'})\cdot e(\mathfrak{c'}/k)\leq 4\cdot 1\Rightarrow e(\mathfrak{K}/k)=4.\] 
Thus, since \[4=e(\mathfrak{K}/k) = e(\mathfrak{K}/\mathfrak{k})\cdot e(\mathfrak{k}/k)=e(\mathfrak{K}/\mathfrak{k})\cdot 4,\] we conclude that $e(\mathfrak{K}/\mathfrak{k})=1$ so that $(k)$ is unramified in $H/K$.
Hence, since $H/K$ is an unramified extension of degree 2,  $H$ is the Hilbert class field of $K$ over $\Q$.
\end{proof}

Once we have these two key lemmas, the proof of Theorem \ref{newclass2} follows exactly as in the proof of Theorem \ref{newclass}. In particular, recall that to satisfy Condition (3) in this proof, the Hilbert class field  of $K$ over $\Q$  must be abelian over $\Q;$ Lemma \ref{hilbert2} implies that $\mathrm{Gal}(H/\Q)\simeq \Z/2\Z\times\Z/4\Z$. Also, note that since $f(K)$ has only two prime divisors in the cyclic case,  the density conditions which guarantee the existence of $s\in\Z$ in the proof of Theorem \ref{newclass} will be met as long as $q,\,k\geq 17$.

\section{Examples}
We now present several new examples of number fields which satisfy the hypotheses of  Theorems \ref{newclass} or \ref{newclass2} and therefore have a Euclidean ideal. Note that the lists of examples presented in Tables 1 and 2 are not exhaustive -- due to space limitations, the author included only a small selection of examples in each case. 

As the conditions in Theorems \ref{newclass} and \ref{newclass2} force the discriminant of $K$ in both cases to be rather large,  these results do not apply to any of the number fields given in the table \cite{RoblotTables} of totally real quartic number fields with class number 2 and discriminant $\leq 600,000$. However, using PARI \cite{PARI2}, we can construct new examples of number fields to which Theorems \ref{newclass} and \ref{newclass2} do apply.\footnote{The code used here can be found at \url{http://blogs.uoregon.edu/catherinehsu/research}.} 

\subsection{Examples in the Biquadratic Case}
We first list all biquadratic number fields of the form $K=\Q(\sqrt{q},\sqrt{k\cdot r}),$ with $29\leq q\leq 41$ and $29\leq k,r\leq 100,$ which satisfy the hypotheses of Theorem \ref{newclass}. When $h_K=2,$ we give explicit minimal polynomials $f(y),\,g(x)$ of $\alpha,\,\beta,$ respectively, where $H = \Q(\sqrt{q},\sqrt{k},\sqrt{r}) = \Q(\alpha)$ and $K =\Q(\beta)$. Indeed, we have \[\begin{split}
f(y)&=\left[\left(y^2-S_1\right)^2-4S_2\right]^2-64S_3y^2,\\g(x)&=\left[x^2-(q+kr)\right]^2-4qkr,
\end{split}\] 
for symmetric polynomials $S_1= q+k+r,\,S_2 = qk+qr+kr,$ and $S_3 = qkr$.

With this information, we can use PARI \cite{PARI2} to find the integers $(s,u)$ required to apply Theorems \ref{growthresult} and \ref{growthresult2} as in the proof of Theorem \ref{newclass} above.

\begin{rmk}
Note that a straightforward application of the algorithm given by Lezowski \cite{lezowski2012examples} determines whether the fields in Table 1 have a norm-Euclidean ideal.
\end{rmk}

\begin{rmk} When $h_K> 2$ in Table 1, we should be able to modify the hypotheses of Theorem \ref{newclass} in order to obtain similar results. However, it is not clear how to modify Condition (3) in the proof of Theorem \ref{newclass} - this modification is necessary because when $h_K>2,$ the ideal $\mathfrak{p}$ being non-principal no longer implies that $[\mathfrak{p}]=[C]$.
\end{rmk}
\LTcapwidth=\textwidth
{\footnotesize
\begin{longtabu}{X[.4,c,m]X[.25,c,m]X[2,c,m]X[.5,c,p]}
	\caption{New examples of non-principal biquadratic number fields with a Euclidean ideal}\\\tabucline[1.5pt]{-}
	$(q,k,r)$ & $h_K$ & Respective minimal polynomials $f(y),\,g(x)$ of $\alpha,\,\beta$ such that $H=\Q(\alpha)$ and $K=\mathbb{Q}(\beta)$ & $(s,u)$\TT\B\\\tabucline[1pt]{-}
	\endfirsthead
	\multicolumn{4}{l}%
	{\tablename\ \thetable\ -- \textit{Continued from previous page}}\\\tabucline[1.5pt]{-}
	$(q,k,r)$& $h_K$ & Respective minimal polynomials $f(y),\,g(x)$ of $\alpha,\,\beta$ such that $H=\Q(\alpha)$ and $K=\mathbb{Q}(\beta)$ & $(s,u)$\TT\B\\\tabucline[1pt]{-}
	\endhead\tabucline[1pt]{-}
	\multicolumn{4}{r}{\textit{Continued on next page}}\ \
	\endfoot\tabucline[1.5pt]{-}
	\endlastfoot
$(29, 37, 41)$&$2$&$y^8-428y^6+38462y^4-1246076y^2+13446889$&$(13,87999)$\T\\&&$x^4-3092x^2+2214144$&\\$(29, 37, 53)$&$16$& &\T\\$(29, 37, 61)$&$2$&$y^8-508y^6+55982y^4-2021356y^2+18207289$&$(23,23)$\T\\&&$x^4-4572x^2+4963984$&\\$(29, 37, 73)$&$2$&$y^8-556y^6+68798y^4-2653948y^2+18003049$&$(5,156663)$\T\\&&$x^4-5460x^2+7139584$&\\$(29, 37, 89)$&$2$&$y^8-620y^6+88574y^4-3778748y^2+14160169$&$(13,191007)$\T\\&&$x^4-6644x^2+10653696$&\\$(29, 37, 97)$&$4$& &\T\\$(29, 41, 53)$&$2$&$y^8-492y^6+51582y^4-1835324y^2+19954089$&$(67,67)$\T\\&&$x^4-4404x^2+4596736$&\\$(29, 41, 61)$&$4$& &\T\\$(29, 41, 73)$&$6$& &\T\\$(29, 41, 89)$&$4$& &\T\\$(29, 41, 97)$&$2$&$y^8-668y^6+103502y^4-4691276y^2+16216729$&$(7,7)$\T\\&&$x^4-8012x^2+15586704$&\\$(29, 53, 61)$&$2$&$y^8-572y^6+70382y^4-2736044y^2+32569849$&$(23,23)$\T\\&&$x^4-6524x^2+10265616$&\\$(29, 53, 73)$&$2$&$y^8-620y^6+83966y^4-3419324y^2+36808489$&$(5,224407)$\T\\&&$x^4-7796x^2+14745600$&\\$(29, 53, 89)$&$4$& &\T\\$(29, 53, 97)$&$4$& &\T\\$(29, 61, 73)$&$2$&$y^8-652y^6+92702y^4-3839644y^2+46063369$&$(7,7)$\T\\&&$x^4-8964x^2+19571776$&\\$(29, 61, 89)$&$2$&$y^8-716y^6+114014y^4-5010524y^2+50055625$&$(7,7)$\T\\&&$x^4-10916x^2+29160000$&\\$(29, 61, 97)$&$2$&$y^8-748y^6+125822y^4-5725756y^2+49378729$&$(7,7)$\T\\&&$x^4-11892x^2+34668544$&\\$(29, 73, 89)$&$12$& &\T\\$(29, 73, 97)$&$2$&$y^8-796y^6+141518y^4-6421708y^2+71284249$&$(5,410703)$\T\\&&$x^4-14220x^2+49730704$&\\$(29, 89, 97)$&$10$& &\T\\$(37, 29, 41)$&$2$&$y^8-428y^6+38462y^4-1246076y^2+13446889$&$(3,3)$\T\\&&$x^4-2452x^2+1327104$&\\$(37, 29, 53)$&$16$& &\T\\$(37, 29, 61)$&$2$&$y^8-508y^6+55982y^4-2021356y^2+18207289$&$(11,11)$\T\\&&$x^4-3612x^2+2999824$&\\$(37, 29, 73)$&$2$&$y^8-556y^6+68798y^4-2653948y^2+18003049$&$(11,11)$\T\\&&$x^4-4308x^2+4326400$&\\$(37, 29, 89)$&$2$&$y^8-620y^6+88574y^4-3778748y^2+14160169$&$(3,3)$\T\\&&$x^4-5236x^2+6471936$&\\$(37, 29, 97)$&$4$& &\T\\$(37, 41, 53)$&$4$& &\T\\$(37, 41, 61)$&$4$& &\T\\$(37, 41, 73)$&$48$& &\T\\$(37, 41, 89)$&$2$&$y^8-668y^6+99662y^4-4668236y^2+35366809$&$(3,3)$\T\\&&$x^4-7372x^2+13046544$&\\$(37, 41, 97)$&$6$& &\T\\$(37, 53, 61)$&$2$&$y^8-604y^6+77198y^4-3425932y^2+49042009$&$(67,67)$\T\\&&$x^4-6540x^2+10214416$&\\$(37, 53, 73)$&$4$& &\T\\$(37, 53, 89)$&$4$& &\T\\$(37, 53, 97)$&$4$& &\T\\$(37, 61, 73)$&$4$& &\T\\$(37, 61, 89)$&$2$&$y^8-748y^6+121982y^4-6163516y^2+80048809$&$(7,7)$\T\\&&$x^4-10932x^2+29073664$&\\$(37, 61, 97)$&$2$&$y^8-780y^6+134046y^4-6970396y^2+81486729$&$(7,7)$\T\\&&$x^4-11908x^2+34574400$&\\$(37, 73, 89)$&$8$& &\T\\$(37, 73, 97)$&$4$& &\T\\$(37, 89, 97)$&$2$&$y^8-892y^6+174254y^4-9443692y^2+152053561$&$(7,7)$\T\\&&$x^4-17340x^2+73891216$&\\$(41, 29, 37)$&$2$&$y^8-428y^6+38462y^4-1246076y^2+13446889$&$(31,31)$\T\\&&$x^4-2228x^2+1065024$&\\$(41, 29, 53)$&$2$&$y^8-492y^6+51582y^4-1835324y^2+19954089$&$(31,31)$\T\\&&$x^4-3156x^2+2238016$&\\$(41, 29, 61)$&$4$& &\T\\$(41, 29, 73)$&$2$&$y^8-572y^6+72302y^4-2839724y^2+22534009$&$(31,31)$\T\\&&$x^4-4316x^2+4309776$&\\$(41, 29, 89)$&$4$& &\T\\$(41, 29, 97)$&$2$&$y^8-668y^6+103502y^4-4691276y^2+16216729$&$(37,230703)$\T\\&&$x^4-5708x^2+7683984$&\\$(41, 37, 53)$&$4$& &\T\\$(41, 37, 61)$&$4$& &\T\\$(41, 37, 73)$&$16$& &\T\\$(41, 37, 89)$&$2$&$y^8-668y^6+99662y^4-4668236y^2+35366809$&$(23,23)$\T\\&&$x^4-6668x^2+10575504$&\\$(41, 37, 97)$&$6$& &\T\\$(41, 53, 61)$&$4$& &\T\\$(41, 53, 73)$&$2$&$y^8-668y^6+95054y^4-4640588y^2+68079001$&$(5,317263)$\T\\&&$x^4-7820x^2+14653584$&\\$(41, 53, 89)$&$6$& &\T\\$(41, 53, 97)$&$2$&$y^8-764y^6+128558y^4-6856172y^2+75394489$&$(5,421567)$\T\\&&$x^4-10364x^2+26010000$&\\$(41, 61, 73)$&$8$& &\T\\$(41, 61, 89)$&$4$& &\T\\$(41, 61, 97)$&$4$& &\T\\$(41, 73, 89)$&$8$& &\T\\$(41, 73, 97)$&$4$& &\T\\$(41, 89, 97)$&$2$&$y^8-908y^6+179102y^4-10388636y^2+182439049$&$(23,23)$\T\\&&$x^4-17348x^2+73822464$&\\\end{longtabu}}

\subsection{Examples in the Cyclic Case}
Now we list all cyclic number fields \[K=\mathbb{Q}\left(\sqrt{q(k+b\sqrt{k})}\right),\] with $17\leq q\leq 41,$ $41\leq k\leq 337,$ and $4\leq b\leq 16,$ which satisfy the hypotheses of Theorem \ref{newclass2}. Again, when $h_K=2,$ we give explicit minimal polynomials $f(y),\,g(x)$ of $\alpha,\,\beta,$ respectively, where $H = \Q(\sqrt{q},\sqrt{k+b\sqrt{k}}) = \Q(\alpha)$ and $K =\Q(\beta)$. Indeed, we have \[\begin{split}
f(y)&=\left[\left(y^2-(q+\alpha_+)\right)^2-4q(\alpha_+)\right]^2\left[\left(y^2-(q+\alpha_-)\right)^2-4q(\alpha_-)\right]^2,\\g(x)&=x^4-2qkx^2+q^2k(k-b^2),
\end{split}\] for $\alpha_+ = k+b\sqrt{k}$ and $\alpha_- =k-b\sqrt{k}$.

With this information, we can use PARI \cite{PARI2} to find the integers $(s,u)$ required to apply Theorems \ref{growthresult} and \ref{growthresult2} as in the proof of Theorem \ref{newclass} above.

\LTcapwidth=\textwidth
{\footnotesize
\begin{longtabu}{X[.4,c,m]X[.25,c,m]X[2,c,m]X[.5,c,p]}
	\caption{New examples of non-principal cyclic number fields with a Euclidean ideal}\\\tabucline[1.5pt]{-}
	$(q,k,b)$ & $h_K$ & Respective minimal polynomials $f(y),\,g(x)$ of $\alpha,\,\beta$ such that $H=\Q(\alpha)$ and $K=\mathbb{Q}(\beta)$ & $(s,u)$\TT\B\\\tabucline[1pt]{-}
	\endfirsthead
	\multicolumn{4}{l}%
	{\tablename\ \thetable\ -- \textit{Continued from previous page}}\\\tabucline[1.5pt]{-}
	$(q,k,b)$& $h_K$ & Respective minimal polynomials $f(y),\,g(x)$ of $\alpha,\,\beta$ such that $H=\Q(\alpha)$ and $K=\mathbb{Q}(\beta)$ & $(s,u)$\TT\B\\\tabucline[1pt]{-}
	\endhead\tabucline[1pt]{-}
	\multicolumn{4}{r}{\textit{Continued on next page}}\ \
	\endfoot\tabucline[1.5pt]{-}
	\endlastfoot
$(17, 41, 4)$&$2$&$y^8-232y^6+13296y^4-159872y^2+6400$&$(5,1399)$\T\\&&$x^4-1394x^2+296225$&\\$(17, 97, 4)$&$2$&$y^8-456y^6+61680y^4-2632832y^2+23503104$&$(3,3)$\T\\&&$x^4-3298x^2+2270673$&\\$(17, 73, 8)$&$2$&$y^8-360y^6+29328y^4-717824y^2+2359296$&$(3,3)$\T\\&&$x^4-2482x^2+189873$&\\$(17, 89, 8)$&$26$& &\T\\$(17, 113, 8)$&$2$&$y^8-520y^6+71568y^4-2998784y^2+3936256$&$(11,11)$\T\\&&$x^4-3842x^2+1600193$&\\$(17, 193, 12)$&$2$&$y^8-840y^6+182768y^4-10233984y^2+10137856$&$(23,23)$\T\\&&$x^4-6562x^2+2733073$&\\$(17, 257, 16)$&$6$& &\T\\$(17, 281, 16)$&$4$& &\T\\$(17, 337, 16)$&$2$&$y^8-1416y^6+533520y^4-46303232y^2+260112384$&$(3,3)$\T\\&&$x^4-11458x^2+7888833$&\\$(29, 17, 4)$&$2$&$y^8-184y^6+8208y^4-102656y^2+16384$&$(19,19)$\T\\&&$x^4-986x^2+14297$&\\$(29, 41, 4)$&$2$&$y^8-280y^6+18576y^4-161024y^2+262144$&$(43,43)$\T\\&&$x^4-2378x^2+862025$&\\$(29, 97, 4)$&$2$&$y^8-504y^6+69648y^4-2268416y^2+9437184$&$(3,3)$\T\\&&$x^4-5626x^2+6607737$&\\$(29, 73, 8)$&$2$&$y^8-408y^6+36144y^4-1051520y^2+7485696$&$(3,3)$\T\\&&$x^4-4234x^2+552537$&\\$(29, 89, 8)$&$2$&$y^8-472y^6+51504y^4-1653632y^2+4393216$&$(17,5179)$\T\\&&$x^4-5162x^2+1871225$&\\$(29, 113, 8)$&$2$&$y^8-568y^6+80304y^4-3255680y^2+30976$&$(11,11)$\T\\&&$x^4-6554x^2+4656617$&\\$(29, 193, 12)$&$2$&$y^8-888y^6+195344y^4-12099840y^2+802816$&$(31,31)$\T\\&&$x^4-11194x^2+7953337$&\\$(29, 257, 16)$&$6$& &\T\\$(29, 281, 16)$&$4$& &\T\\$(29, 337, 16)$&$10$& &\T\\$(37, 17, 4)$&$10$& &\T\\$(37, 41, 4)$&$2$&$y^8-312y^6+23056y^4-188672y^2+409600$&$(5,3039)$\T\\&&$x^4-3034x^2+1403225$&\\$(37, 97, 4)$&$2$&$y^8-536y^6+75920y^4-2016512y^2+4194304$&$(31,31)$\T\\&&$x^4-7178x^2+10756233$&\\$(37, 73, 8)$&$2$&$y^8-440y^6+41648y^4-1280384y^2+11397376$&$(19,19)$\T\\&&$x^4-5402x^2+899433$&\\$(37, 89, 8)$&$2$&$y^8-504y^6+57520y^4-1864064y^2+8952064$&$(5,6591)$\T\\&&$x^4-6586x^2+3046025$&\\$(37, 113, 8)$&$2$&$y^8-600y^6+87088y^4-3407744y^2+2119936$&$(13,8375)$\T\\&&$x^4-8362x^2+7580153$&\\$(37, 193, 12)$&$26$& &\T\\$(37, 257, 16)$&$6$& &\T\\$(37, 281, 16)$&$2$&$y^8-1272y^6+379696y^4-26813312y^2+153760000$&$(5,20799)$\T\\&&$x^4-20794x^2+9617225$&\\$(37, 337, 16)$&$2$&$y^8-1496y^6+566960y^4-56650112y^2+13897984$&$(43,43)$\T\\&&$x^4-24938x^2+37369593$&\\$(41, 17, 4)$&$2$&$y^8-232y^6+14064y^4-248960y^2+92416$&$(19,19)$\T\\&&$x^4-1394x^2+28577$&\\$(41, 97, 4)$&$2$&$y^8-552y^6+79344y^4-1892480y^2+2509056$&$(3,3)$\T\\&&$x^4-7954x^2+13207617$&\\$(41, 73, 8)$&$2$&$y^8-456y^6+44688y^4-1401344y^2+13307904$&$(3,3)$\T\\&&$x^4-5986x^2+1104417$&\\$(41, 89, 8)$&$10$& &\T\\$(41, 113, 8)$&$2$&$y^8-616y^6+90768y^4-3482624y^2+4194304$&$(11,11)$\T\\&&$x^4-9266x^2+9307697$&\\$(41, 193, 12)$&$2$&$y^8-936y^6+209648y^4-13843584y^2+21977344$&$(67,67)$\T\\&&$x^4-15826x^2+15897217$&\\$(41, 257, 16)$&$6$& &\T\\$(41, 281, 16)$&$2$&$y^8-1288y^6+386064y^4-28725248y^2+205520896$&$(7,7)$\T\\&&$x^4-23042x^2+11809025$&\\$(41, 337, 16)$&$2$&$y^8-1512y^6+574224y^4-58626560y^2+1806336$&$(3,3)$\T\\&&$x^4-27634x^2+45886257$&\\\end{longtabu}}

\bibliographystyle{plain}
\bibliography{bib}{}

\begin{thebibliography}{10}

\bibitem{graves2011has}
Hester Graves.
\newblock $\mathbb{Q}(\sqrt{2},\sqrt{35})$ has a non-principal {Euclidean}
  ideal.
\newblock {\em International Journal of Number Theory}, 7(08):2269--2271, 2011.

\bibitem{graves2013growth}
Hester Graves.
\newblock Growth results and {Euclidean} ideals.
\newblock {\em Journal of Number Theory}, 133(8):2756--2769, 2013.

\bibitem{graves2013family}
Hester Graves and M~Murty.
\newblock A family of number fields with unit rank at least 4 that has
  {Euclidean} ideals.
\newblock {\em Proceedings of the American Mathematical Society},
  141(9):2979--2990, 2013.

\bibitem{PARI2}
The~PARI Group.
\newblock Pari/gp version {\tt 2.7.0}, jun 2014.

\bibitem{mharper1}
M.~Harper.
\newblock $\mathbb{Z}[\sqrt{14}]$ is {Euclidean}.
\newblock {\em Canad. J. Math.}, 56:55--70, 2004.

\bibitem{harper2004euclidean}
Malcolm Harper and M~Ram Murty.
\newblock Euclidean rings of algebraic integers.
\newblock {\em Canadian Journal of Mathematics}, 56(1):71--76, 2004.

\bibitem{heathbrown}
D.R. Heath-Brown.
\newblock Artin's conjecture for primitive roots.
\newblock {\em Quart. J. Math. Oxford Ser. (2)}, 37:27--38, 1986.

\bibitem{lenstra1979euclidean}
H.K. Lenstra.
\newblock Euclidean ideal classes.
\newblock {\em Ast\'{e}risque}, 61:121--131, 1979.

\bibitem{lezowski2012examples}
Pierre Lezowski.
\newblock Examples of norm-{Euclidean} ideal classes.
\newblock {\em International Journal of Number Theory}, 8(05):1315--1333, 2012.

\bibitem{motzkin1949euclidean}
Th~Motzkin.
\newblock The {Euclidean} algorithm.
\newblock {\em Bulletin of the American Mathematical Society},
  55(12):1142--1146, 1949.

\bibitem{narkiewicz1988units}
Wladyslaw Narkiewicz.
\newblock Units in residue classes.
\newblock {\em Archiv der Mathematik}, 51(3):238--241, 1988.

\bibitem{RoblotTables}
Xavier-Francois Roblot.
\newblock Table of {Hilbert} class field of totally real fields of degree
  2,3,4, 2013.

\bibitem{spearman1997conductor}
Blair~K Spearman and Kenneth~S Williams.
\newblock The conductor of a cyclic quartic field using {Gauss} sums.
\newblock {\em Czechoslovak Mathematical Journal}, 47(3):453--462, 1997.

\end{thebibliography}
\end{document}